\documentclass[11pt]{article}

\usepackage{algorithm}
\usepackage{algorithmicx}
\usepackage{algpseudocode}
\usepackage{amssymb,amsthm}
\usepackage{graphicx}
\usepackage{xcolor}
\newtheorem{theorem}{Theorem}
\newtheorem{prop}[theorem]{Proposition}
\newtheorem{cor}[theorem]{Corollary}
\newtheorem{remark}{Remark}
\newtheorem{conj}{Conjecture}

\newcommand{\bit}{\begin{itemize}}
\newcommand{\eit}{\end{itemize}}
\newcommand{\ben}{\begin{enumerate}}
\newcommand{\een}{\end{enumerate}}

\newcommand{\bc}{\begin{center}}
\newcommand{\ec}{\end{center}}

\newcommand{\ds}{\displaystyle }

\newcommand{\gl}{\gamma_{\stackrel{}{l}}}
\newcommand{\il}{\iota_{\stackrel{}{l}}}
\newcommand{\al}{\alpha}

\title{Independent Locating-Dominating Sets in Pseudotrees}
\author{J. C\'{a}ceres\footnote{CDTIME and Departamento de Matem\'aticas, Universidad de Almer\'ia (\texttt{jcaceres@ual.es}).} \and I. M. Pelayo\footnote{Departament de Matem\`atiques, Universitat Polit\`ecnica de Catalunya (\texttt{ignacio.m.pelayo@upc.edu}).}}
\date{}

\begin{document}
\maketitle

\begin{abstract}
An ILD-set in a connected graph is a subset $S$ of vertices such that it is both independent and locating-dominating. The independent locating-dominating number of a graph G is the minimum cardinality of an ILD-set set of $G$.

A well-known fact is that any graph with girth at least 5 has an ILD-set, but that is not clear for graphs with girth 3 and 4. 
In this work, we prove that there are graphs with no ILD-sets  for any order $n\geq 9$ and girth 4, also showing some sufficient conditions for a bipartite graph to contain an ILD-set.

Moreover, we focus  our attention on  trees and on unicyclic graphs, showing that every tree and every unicyclic graph contains ILD-sets, whenever in the latter case,  it is twin-free. 
Finally, a number of bounds, realization theorems and algorithms to find an ILD-set in those families of graphs are provided.
\end{abstract}


\section{Introduction}\label{sec1:intro}

All the graphs considered are undirected, simple, finite and (unless otherwise stated) connected.
The \emph{order} $n$ and \emph{size} $m$ of a graph $G=(V,E)$ are respectively the number of vertices and the number of edges. Unless otherwise specified, we will set $V=[n]=\{1,\ldots,n\}$.

The \emph{open neighborhood} of a vertex $v$ in a graph $G$ is $\displaystyle N(v)=\{w \in V(G) :vw \in E(G)\}$, and the \emph{closed neighborhood} of $v$ is $N[v]=N(v)\cup \{v\}$. 
If two vertices $u,v \in V(G)$ have the same open or closed neighborhood, i.e. $N(u)=N(v)$ or $N[u]=N[v]$, then they are called \emph{twins}, being \emph{true twins} when they agree in the closed neighborhood. 

Let $W\subseteq V(G)$ be a subset of vertices of  $G$.
The  \emph{open neighborhood} of $W$ is the set $N(W)=\cup_{v\in W} N(v)$.
The subgraph of $G$ induced by $W$, denoted by $G[W]$, has $W$ as vertex set and $E(G[W]) = \{vw \in E(G) : v \in W,w \in W\}$. If $G[W]$ is either a complete graph or the empty graph, then it is said to be either a \emph{clique} or an \emph{independent set} (and also an \emph{stable} set) of $G$.

The \emph{degree} of $v$ is $\deg(v)=|N(v)|$.
The minimum degree  (resp. maximum degree) of $G$ is $\delta(G)=\min\{\deg(u):u \in V(G)\}$ (resp. $\Delta(G)=\max\{\deg(u):u \in V(G)\}$).
If $\deg(v)=1$, then $v$ is said to be a  \emph{leaf} of $G$ and
the set and the number of leaves of $G$ are denoted by ${\cal L}(G)$ and $\ell(G)$, respectively. 
A \emph{support vertex} (resp. \emph{strong support vertex}) is a vertex adjacent to a leaf (resp., to at least two leaves).

As usual, we  denote by $K_n$, $P_n$, $W_n$ and $C_n$, respectively, the complete graph, path, wheel and cycle of order $n$. $K_{r,s}$ denotes the complete bipartite graph whose disjoint sets are $\bar{K}_r$ and $\bar{K}_s$, respectively.
In particular, $K_{1,n-1}$ is the star with $n-1$  leaves.
A split graph, or $S_{r,s}$, is a graph whose vertices can be partitioned into a clique $K_r$ and a maximum independent vertex set $\bar{K}_s$.

Let $G$ be a graph whose vertices represent the rooms and corridors of a facility. Suppose this facility must be guarded against intruders, so it is necessary to install cameras on a set of vertices $D$. To avoid wasting resources, no two adjacent vertices should both host a camera; therefore, the cameras are placed on non-adjacent vertices. Thus, the set of locations with cameras must form an independent set of the graph. The maximum number of independent vertices in $G$ is usually denoted by $\alpha(G)$. The concept of independence in graphs is central to graph theory and is related to many other parameters, such as the vertex cover, the chromatic number, the edge cover, and the clique number, among others.

Now, assume that our cameras can protect not only the vertices on which they are placed but also the adjacent vertices. 
Hence, we require $D$ to \emph{dominate} the graph. 
The cardinality of a minimum dominating set in $G$ is denoted by $\gamma(G)$. 
Note that a set is  maximal independent set if it is both independent and dominating set, and  so, in particular,  $\gamma(G) \leq \alpha(G)$. 
Vertex domination has been studied since the 1950s and is one of the most prolific concepts in graph theory. 
We refer the reader to~\cite{hhs98} for an excellent introduction to the topic.

A set $D\subseteq V(G)$ is  an \emph{independent dominating set} of $G$ if it is both independent and dominating. 
The \emph{independent domination number}, denoted by  $\iota(G)$,   is the  minimum cardinality of an independent dominating set.
This concept was introduced in~\cite{b62}, and in~\cite{o62}, and since then, it has been thoroughly studied in the literature (\cite{al78, bphm77, bc79, c23,cns24,ch74,ch77,f92}), The reader can check~\cite{gh13} for a survey on this topic.

Once an intruder is detected by our cameras, they must be located precisely and without error so that they can be neutralized as soon as possible. That is, depending on which subset of cameras detects the intruder, we must be able to determine exactly where they are located. For this reason, we require $D$ to be locating-dominating.

Introduced in~\cite{s87,s88}, a set $D\subseteq V(G)$ is  a \emph{locating-dominating set} of $G$ if every two vertices  $u,v\in V(G)\setminus D$ verify $$\displaystyle \emptyset\neq N(u)\cap D\neq N(v)\cap D\neq\emptyset.$$ 
The \emph{location-domination number}  $\gamma_{\stackrel{}{l}}(G)$   is the  minimum cardinality of a locating-dominating set. 
A locating-dominating set of cardinality $\gamma_{\stackrel{}{l}}(G)$ is called a \emph{$\gl$-set}. 
Many authors have been interested in this concept (see~\cite{bchl04,bclm11,bcmms07,chmpp13,cms10, hmp14,hmp19}). A survey can be found in~\cite{lhc20}, and a practical detailed list of references on this subject is in~\cite{j26}.

In~\cite{ss18}, and independently in~\cite{wdak17}, decided to combine those concepts and introduced the \emph{independent locating-dominating sets}, the \emph{ILD-sets} for short, as those sets 
which are both independent and locating-dominating. 
Other authors followed the idea (see~\cite{daw20,kkk25,xcmfl25}) and denoted  the \emph{independent location domination number}  $\il(G)$  as the  minimum cardinality of an ILD-set.

Notice that not every graph contains ILD-sets, being the complete graph family the simplest example.
An \emph{ILD graph} is a graph containing ILD-sets.
Otherwise, it is called a non-ILD graph.
Observe also that, for every ILD graph $G$, 
$$\gamma(G) \le \min\{\iota(G),\gamma_{\stackrel{}{l}}(G)\} \le \max\{\iota(G),\gamma_{\stackrel{}{l}}(G)\} \le \il(G)
\le \alpha(G).$$

In this work, we continue the study of ILD-sets and, in addition to presenting several results for general graphs, we focus specifically on pseudotrees, namely, trees and unicyclic graphs. A graph $G$ of order $n$ and size $m$ is called \emph{unicyclic} if $n=m$.
Let  $C_g$ be its unique cycle.
The connected component of $G-E(C_g)$ containing a vertex $v\in V(C_g)$  is denoted by $T_v$ and it is called the \emph{branching tree} of $v$.
The tree $T_v$ is said to be \emph{trivial} if $V(T_v)=\{v\}$.

For additional details and information on basic graph theory we refer the reader to~\cite{clz16}.

The paper is organized as follows: in Section~\ref{nr} we continue with some interesting results found in the literature that will be useful in the rest of the paper and show some new results for general graphs. In Section~\ref{nrt}, we turn out our attention to trees improving a previous bound of $\il$. Unicyclic graphs are studied in Section~\ref{unic}, where the conditions for such a graph to have an ILD-set are proved. In Section~\ref{algo}, we introduce two algorithms for computing an ILD-set both in trees and unicyclic graphs that are based on the previous results. We close the work with a section on conclusions and further work.

\section{General results}\label{nr}
In this section, several relevant results for general graphs are introduced. In~\cite{cms10}, the authors proved the following result for trees:

\begin{theorem} [\cite{cms10}]
If  $T$ is a tree, then $\iota(T)\le \gl(T)$.
\label{iogltrees}
\end{theorem}

Similarly, a new inequality involving $\iota(G)$ and $\gl(G)$ is established for general graphs.

\begin{theorem}
\label{iogl}
For every graph  $G$,  $\iota(G)\le 2 ^{\gl(G)-1}$. 
Moreover, the equality is achieved.
\end{theorem}
\begin{proof}
Let $D=U_r \cup U_s$ be a $\gl$-set of $G$ such that $|D|=r+s$ and $U_r=\bar{K}_r$ is a maximum independent set of $D$.
Let $V\setminus D= W_h\cup W_k$ such that $|V\setminus D|=h+k$ and  $N(U_r)\cap (V\setminus D)=W_h$.
Let $S$ a maximal independent set of $W_k$.

Notice that $\Omega=U_r \cup S$ is an independent dominating set of $G$.
Observe also that $k\le 2^s-1$, since $U_s$ is a locating-dominating set of $W_k$.
Hence,  
$$\iota(G)\le |\Omega|=r+|S| \le r+k \le r+2^s-1\le 2^{r+s}-1 =2 ^{\gl(G)-1}.$$

To prove that the equality $\iota(G)= 2 ^{\gl(G)-1}$ holds, consider the following family 
$\{\Gamma_r\}_{r\ge1}$ of split graphs·
For every integer $r\ge1$, $\Gamma_r=(V_r,E_r)$ is the graph such that $V_r=U_r\cup W_r$, $U_r$ is the set $[r]$, 
$W_r=\{w_1,\ldots,w_s\}$ is the set of order $s=2^{r}-1$: ${\cal P}([r])\setminus\{\emptyset\}$, 
$E_r=H_r \cup K_r$, 
$H_r=\{ij \,:\, i,j\in U_r\}$ and
$K_r=\{iw_j: i\in w_j\}$ (see Figure  \ref{fig16}).
Check that $\gl(\Gamma_r)=|U_r|=r$ and $\iota(\Gamma_r)=|W_r|=2^{r}-1$.
\end{proof}

\begin{figure}[hbt]
  \centering
        \includegraphics[width=.95\textwidth]{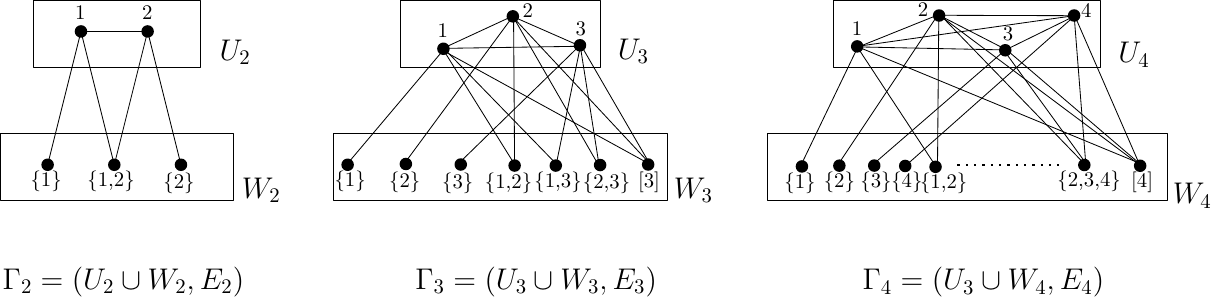}
  \caption{Split graphs $\Gamma_2$, $\Gamma_3$ and $\Gamma_4$.}
  \label{fig16}
\end{figure}

Note that ILD graphs are very dependent of their girth, as the next theorem states.
\begin{theorem}[\cite{cms10}]
\label{ildg5}
If $G$ is a graph with girth $g\ge 5$, then every maximum independent set $S$ is an ILD-set. 
\end{theorem}
Hence, in studying whether or not a graph has an ILD-set, one should focus on graphs with girth 3 and 4. Table~\ref{tabla1} shows the $\il$ number of some well-known families of graphs (the solid triangle meaning that there is no ILD-graph in that family), and its relation with other domination and independence parameters.


\begin{table}[hbt] 
\caption{Five parameters of some basic families.}\label{tabla1} 
\begin{center}
\begin{tabular}{cccccc}
\noalign{\hrule height 1.5pt}
 $G$                 & $\gamma$   &  $\iota$   &   $\gl$  &  $\il$  & $\al$  \\  \hline
 $P_n$   &  $\lceil\frac{n}{3}\rceil$     & $\lceil\frac{n}{3}\rceil$ & $\lceil\frac{2n}{5}\rceil$  & $\lceil\frac{2n}{5}\rceil$ & $\lceil\frac{n}{2}\rceil$ \\ 
 $C_4$   &  2  & 2 & 2  & {\bf\Large$\color{blue}\blacktriangledown$} & 2 \\
 $C_n$, $n\ge5$    & $\lceil\frac{n}{3}\rceil$  &  $\lceil\frac{n}{3}\rceil$  & $\lceil\frac{2n}{5}\rceil$  & $\lceil\frac{2n}{5}\rceil$  & $\lfloor\frac{n}{2}\rfloor$ \\
 $K_n$, $n\ge3$    &   1   &  1  &  $n-1$ &  {\bf\Large$\color{blue}\blacktriangledown$}  & 1  \\
 $K_{1,n-1}$       &      1     &   1    &   $n-1$   &  $n-1$    &  $n-1$  \\
 $K_{r,n-r}$, $1<r\le n-r$ &      2     &   2   &    $n-2$    &  {\bf\Large$\color{blue}\blacktriangledown$}   & $n-r$  \\
 $W_{n}$, $n\ge7$   &  1  & 1  & $\lceil\frac{2n-2}{5}\rceil$ & $\lceil\frac{2n-2}{5}\rceil$ & $\lfloor\frac{n}{2}\rfloor$  \\
\noalign{\hrule height 1.5pt}
\end{tabular}
\end{center}
\end{table}


Some important facts to take into account for the rest of the work are the following:
\begin{remark}
{\color{white}.}

\ben

\item
Every graph   that is either  a  complete graph, the cycle $C_4$, a  complete bipartite graph 
or a split graph with true twins is a non-ILD graph.

\item
A split graph $S_{r,s}$ is an ILD graph if and only if it has no true twins.
In such a case, its maximum independent  set $\bar{K}_s$ is its unique ILD-set.

\item
There are  11 non-ILD graphs of order at most 5 (see Figure  \ref{fig15}).
Notice that except the wheel $W_{5}$, each of them is either a complete graph, a  complete bipartite graph 
or a split graph with true twins.

\item
There are  11 non-ILD graphs of order 6, without twins (see Figure  \ref{fig8}).
Notice that all of them have girth $g=3$.

\een
\end{remark}

\begin{figure}[hbt]
  \centering
        \includegraphics[width=.85\textwidth]{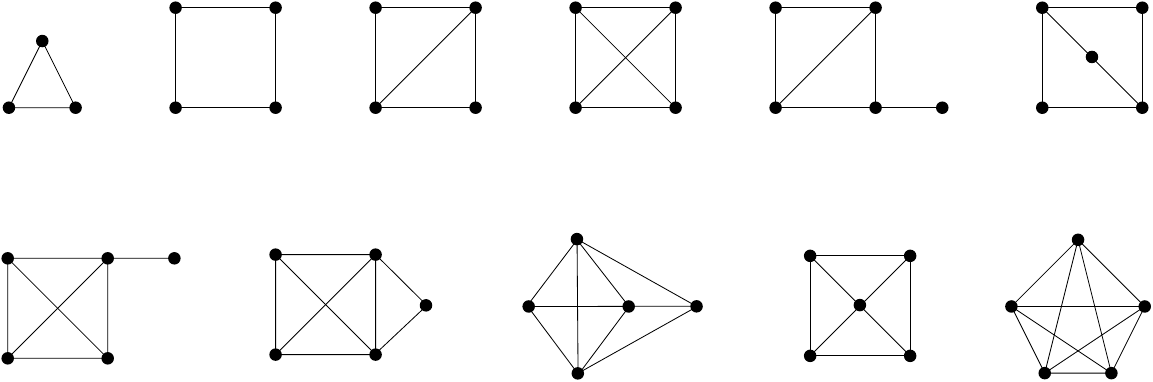}
  \caption{There are  11 non-ILD graphs of order at most 5.}
  \label{fig15}
\end{figure}

\begin{figure}[hbt]
  \centering
        \includegraphics[width=.85\textwidth]{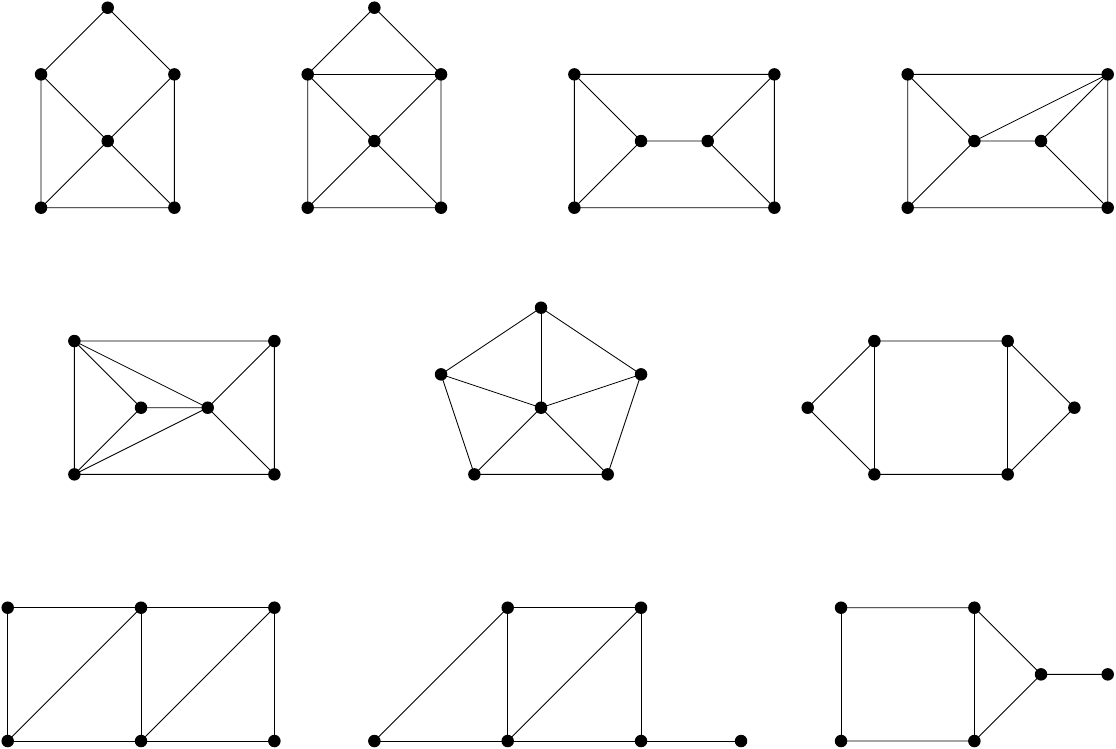}
  \caption{There are  11 non-ILD graphs of order 6, without twins.}
  \label{fig8}
\end{figure}

As we have seen above, for a graph to have an ILD-set is affected by two main factors: its girth and containing  twins or not. Moreover, if $\alpha(G)$ is small enough, then $G$ does not have ILD-sets.

\begin{prop}
Let $G$ be a graph of order $n$ such that $\al(G)=k$.
If $n \ge k + 2^k$, then $G$ is a non-ILD graph.
\label{nild2}
\begin{proof}
Suppose that $G$ is an ILD graph.
Let $S$ be an ILD-set of $G$.
Notice that $|S|\le k$ and $|V-S|\ge n-k$.
Hence, $n-k \le |V-S| \le 2^k -1$, since $S$ is locating-dominating.
In consequence, $n \le k +2^k -1$. 
\end{proof}
\end{prop}

Some conditions for a bipartite graph to have  ILD-sets are showed in the next result.

\begin{prop}
\label{prop:bipartite}
Let $G=(U\cup W,E)$ be a bipartite graph of order $n$, being  $U$ and $W$ the stable parts of $G$.
Then,
\begin{enumerate}
\item $W$ is an ILD-set of $G$ if and only if $U$ has no twins.
\item $U$ is an ILD-set of $G$  if and only if $W$ has no twins.
\item If $G$ has no twins, then $G$ is an ILD graph.
\end{enumerate}
\begin{proof}
To prove {\rm \bf(1)}, it is enough to notice that if  $x, y \in U$, then they are twins
if and only if $N(x)\cap W =N(x)\cap W $.
Item {\rm \bf(2)} is proved similarly and item {\rm \bf(3)}  is a  corollary of both items.
\end{proof}
\label{bipild}
\end{prop}

\begin{figure}[hbt]
  \centering
        \includegraphics[width=0.85\textwidth]{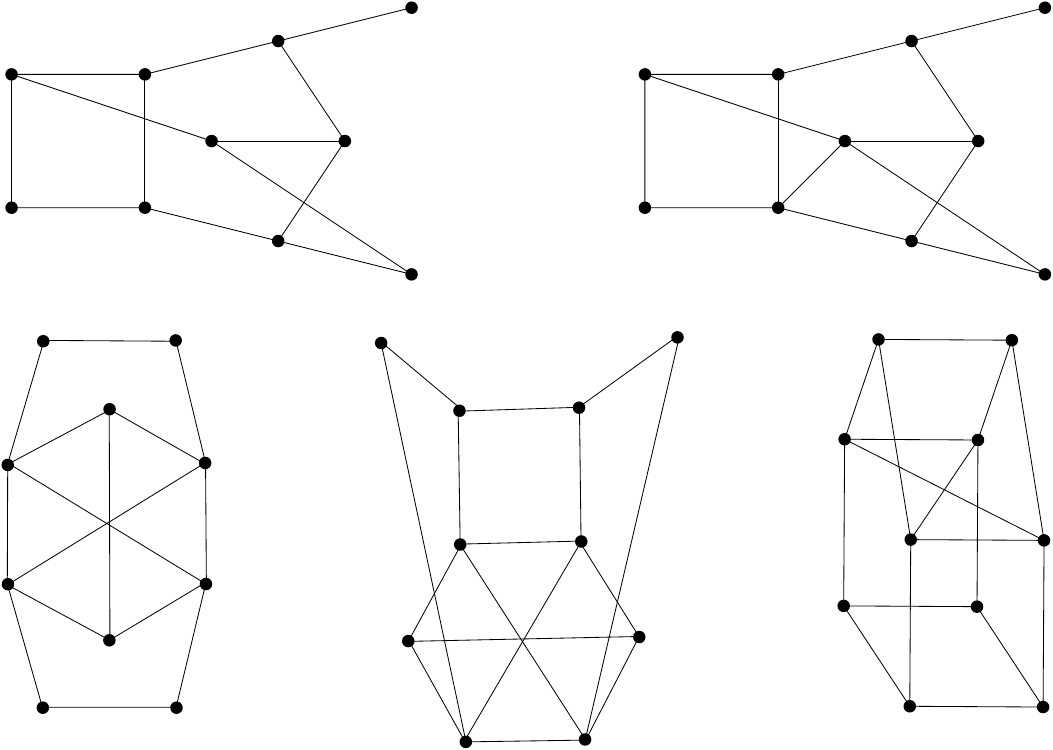}
  \caption{There are five twin-free graphs of order $n=10$  and  girth $g=4$, containing no ILD-set.}
  \label{fig2}
\end{figure}

The next proposition provides the ILD-graphs for twin-free graphs of girth 4, and order  $n\leq 10$. It has been proved by checking all the cases with the aid of computer calculations.

\begin{prop}
Let  $G$ be a twin-free graph of order $n$ and girth $g=4$. 

\begin{enumerate}

\item
If $n<9$, then $G$ is an ILD graph.

\item
If $n=9$ and $G$ is not an ILD graph, then it is isomorphic to the graph  displayed  in Figure \ref{fig4} (left).

\item
If $n=10$ and $G$ is not an ILD graph, then it is isomorphic to some of the graphs displayed  in Figure \ref{fig2}.

\end{enumerate}
\label{noild}

\end{prop}

To finish this section, we prove that there always exists a non ILD-graph of girth 4 with order $n\geq 9$. 

\begin{prop}
For every integer $n\ge9$, there is  a  non-ILD twin-free graph  $G$ of order $n$ and girth $g=4$. 

\begin{proof}
Clearly, for the case $n=9$ we can consider the graph $H_9$ in Figure~\ref{fig4}. 
It is also not difficult to check that for $n\geq 10$ the graph $H_n$ (same Figure~\ref{fig4}) is not an ILD graph.
\end{proof}
\end{prop}

\begin{figure}[h]
  \centering
        \includegraphics[width=0.85\textwidth]{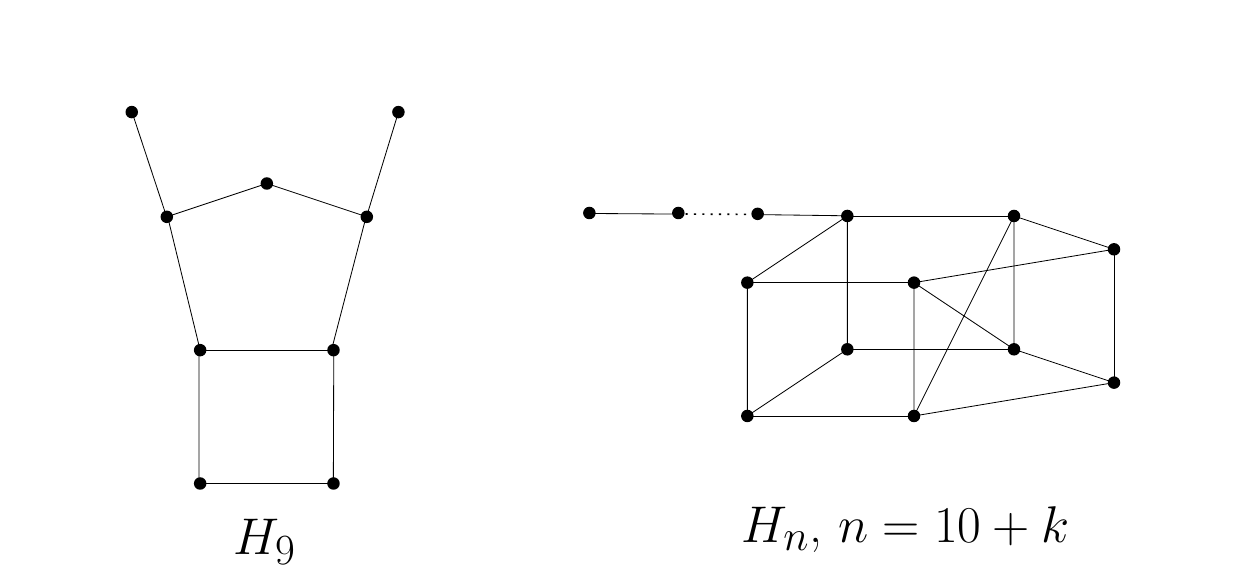}
  \caption{Left: Non-ILD graph of order $n=9$. Right: Non-ILD family of graphs of order $n\ge10$.}
  \label{fig4}
\end{figure}

\section{Trees}\label{nrt}

Any tree has an ILD-set, so it is natural to begin with them to give bounds for, and compute the independent locating-dominating number whenever it is possible. 



The next result proves the values of the independent dominating number of trees with a few vertices and paths of any length.

\begin{prop}[\cite{ss18}]
Let $T$  be a tree of order $n$.
Then,

\ben

\item
If $n\le6$, then  $\gl(T)=\il(T)$, except for the tree displayed in Figure \ref{fig1}, in which case 
$\il(T)=\gl(T)+1=4$.

\item
If $T$ is the path $P_n$, then $\gl(P_n)=\il(P_n)=\lceil\frac{2n}{5}\rceil$.

\een 
\label{small1}
\end{prop}

\begin{figure}[h]
  \centering
        \includegraphics[width=.85\textwidth]{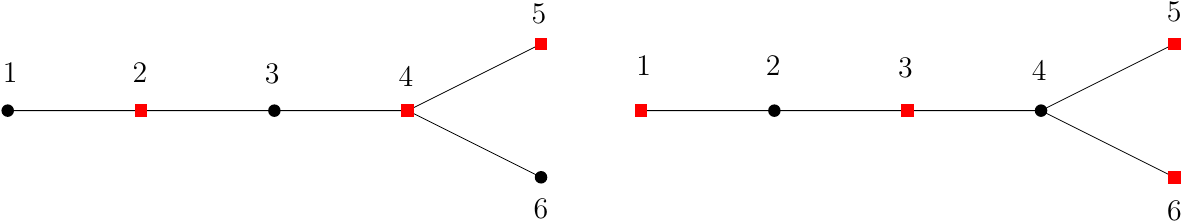}
  \caption{The set $\{2,4,5\}$ is a $\gl$-set, meanwhile the set $\{1,3,5,6\}$ is an $\il$-set.}
  \label{fig1}
\end{figure}

In~\cite{s87,ss18}, the authors proved that for any tree of order $n$,
$$\ds \frac{n}{3} < \gl(T) \le \il(T) \le 2\gl(T) - 1$$
\label{ildtrees1}
However, it is possible to improve the previous upper bound in one unit, being in this case tight.

\begin{theorem}
Let $T$ be a tree  of order $n\ge 3$.
Then,
$$\ds \gl(T) \le \il(T) \le 2\gl(T)-2.$$
\label{nildtrees1}
\begin{proof}
Clearly, as a direct consequence of both definitions, $\ds \gl(T) \le \il(T)$, being this bound tight since 
$\gl(P_n)=\il(P_n)=\lceil\frac{2n}{5}\rceil$ (see Proposition \ref{small1}).

To prove that $\il(T) \le 2\gl(T)-2$, we proceed by induction on  $n$.

According to Proposition \ref{small1}, if either $n\le5$ or $T$ is the path $P_n$, then $\ds \gl(T) = \il(T)$.
Assume thus that $T$ is a tree of order $n\ge6$ containing at least one vertex $u$ of degree $\deg(u)\ge3$.

\begin{figure}[hbt]
  \centering
        \includegraphics[width=1\textwidth]{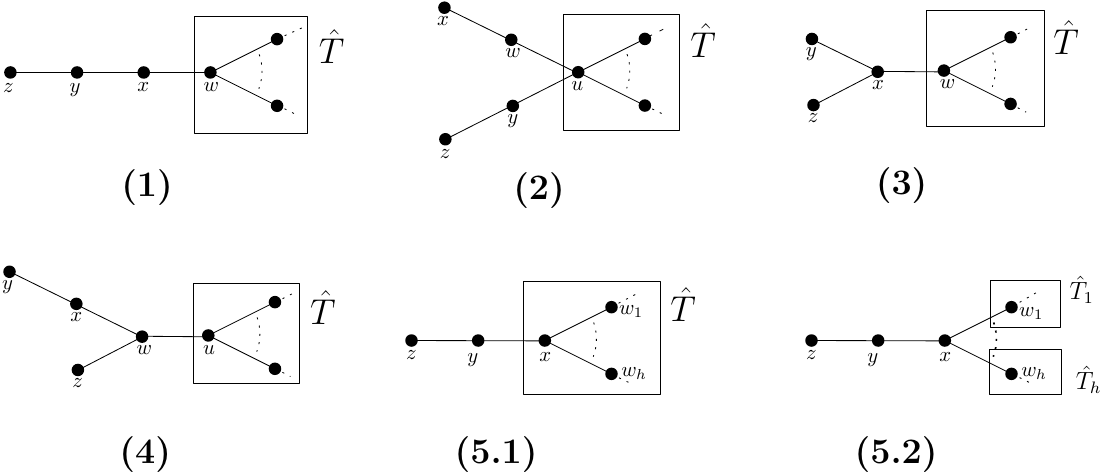}
  \caption{Every tree $T$ contains a subtree $\hat{T}$ such that $\gl(\hat{T}) \le \gl(T) -1$ and 
$\il(T) \le \il(\hat{T}) +2$.}
  \label{fig7}
\end{figure}

We claim that there is a subgraph $\hat{T}$ of $T$ satisfying the following two conditions:

\ben

\item[\bf (i)]
$\gl(\hat{T}) \le \gl(T) -1$

\item[\bf (ii)]
$\il(T) \le \il(\hat{T}) +2$

\een
 
To prove this statement, we distinguish cases.

\noindent
{\bf Case (1):} 
There is a 4-vertex set $\{w,x,y,z\}$ inducing the path $P_4$, such that $\deg(w)\ge2$, $\deg(x)=\deg(y)=2$ and
$\deg(z)=1$. (see Figure \ref{fig7} (1)).
Take the tree $\hat{T}=T-\{x,y,z\}$.

\bit

\item
Let $S$ be a $\gl$-set of $T$.
If $x\not\in S$, (resp., $x\in S$), then take the set $\hat{S}=S\cap V(\hat{T})$ (resp., $\hat{S}=S\cap V(\hat{T})\cup \{w\}$).
Check that in both cases, $\hat{S}$ is an LD-set of $\hat{T}$ and $|\hat{S}|\le |S|-1$.
Hence, $\gl(\hat{T}) \le \gl(T) -1$.

\item
Let $\hat{S}$ be an $\il$-set of $\hat{T}$.
If $w\not\in \hat{S}$, (resp., $w\in \hat{S}$), then take the set $S=\hat{S}\cup \{x,z\}$ 
(resp., $S=\hat{S}\cup \{y\}$).
Check that in both cases, $S$ is an ILD-set of $T$ and $|S|\le |\hat{S}|+2$.
Hence, $\il(T) \le \il(\hat{T}) +2$.

\eit

\noindent
{\bf Case (2):} 
There is a 5-vertex set $\{x,w,u,y,z\}$ inducing the path $P_5$, such that $\deg(w)=\deg(y)=2$, $\deg(x)=\deg(z)=1$ and
$\deg(u)\ge 3$. (see Figure \ref{fig7} (2)).
Take the tree $\hat{T}=T-\{x,w, y,z\}$.

\bit

\item
Let $S$ be a $\gl$-set of $T$.
If $\{w,y\}\cap S = \emptyset$, (resp.,  $\{w,y\}\cap S \neq \emptyset$), then take the set $\hat{S}=S\cap V(\hat{T})$ (resp., $\hat{S}=S\cap V(\hat{T})\cup \{u\}$).
Check that in both cases, $\hat{S}$ is an LD-set of $\hat{T}$ and $|\hat{S}|\le |S|-1$.
Hence, $\gl(\hat{T}) \le \gl(T) -1$.

\item
Let $\hat{S}$ be an $\il$-set of $\hat{T}$.
Take the set $S=\hat{S}\cup \{x,z\}$.
Check that  $S$ is an ILD-set of $T$ and $|S|= |\hat{S}|+2$.
Hence, $\il(T) \le \il(\hat{T}) +2$.

\eit

\noindent
{\bf Case (3):} 
There is a 4-vertex set $\{w,x,y,z\}$ inducing the star $K_{1,3}$, such that $\deg(y)=\deg(z)=1$, $\deg(x)=3$ and
$\deg(w)\ge 2$. (see Figure \ref{fig7} (3)).
Take the tree $\hat{T}=T-\{x,y,z\}$.


\bit

\item
Let $S$ be a $\gl$-set of $T$.
If $x\not\in S$, (resp., $x\in S$), then take the set $\hat{S}=S\cap V(\hat{T})$ (resp., $\hat{S}=S\cap V(\hat{T})\cup \{w\}$).
Check that in both cases, $\hat{S}$ is an LD-set of $\hat{T}$ and $|\hat{S}|\le |S|-1$.
Hence, $\gl(\hat{T}) \le \gl(T) -1$.

\item
Let $\hat{S}$ be an $\il$-set of $\hat{T}$.
Notice that for every ILD-set $S$ of $T$, $|\{x,y,z\}\cap S|\ge 2$. 
Take the set $S=\hat{S}\cup \{x,y\}$
Check that  $S$ is an ILD-set of $T$ and $|S|=|\hat{S}|+2$.
Hence, $\il(T) \le \il(\hat{T})+2$.

\eit

\noindent
{\bf Case (4):} 
There is a 4-vertex set $\{y,x,w,z\}$ inducing the path $P_{4}$, such that $\deg(y)=\deg(z)=1$, $\deg(x)=2$ and
$\deg(w)=3$. (see Figure \ref{fig7} (4)).
Take the tree $\hat{T}=T-\{y,x,w,z\}$.

\bit

\item
Let $S$ be a $\gl$-set of $T$.
If $w\not\in S$, (resp., $w\in S$), then take the set $\hat{S}=S\cap V(\hat{T})$ (resp., $\hat{S}=S\cap V(\hat{T})\cup \{u\}$).
Check that in both cases, $\hat{S}$ is an LD-set of $\hat{T}$ and $|\hat{S}|\le |S|-1$.
Hence, $\gl(\hat{T}) \le \gl(T) -1$.

\item
Let $\hat{S}$ be an $\il$-set of $\hat{T}$.
Take the set $S=\hat{S}\cup \{x,z\}$
Check that  $S$ is an ILD-set of $T$ and $|S|=|\hat{S}|+2$.
Hence, $\il(T) \le \il(\hat{T})+2$.

\eit


\noindent
{\bf Case (5):} 
There is a 3-vertex set $\{x,u,z\}$ inducing the path $P_{3}$, such that $\deg(z)=1$, $\deg(y)=2$,
$\deg(x)=h+1\ge 3$ and $N(x)=\{y,w_1,\ldots,w_h\}$. (see Figures \ref{fig7} (5.1) and (5.2)).


Let $S$ be a $\gl$-set of $T$.
Notice that $|\{y,z\}\cap S|\ge 1$.
We distinguish two subcases.


\noindent
{\bf Subcase (5.1):}  
$x\in S$.
Take the  tree $\hat{T}=T-\{y,z\}$
(see Figure \ref{fig7} (5.1)).
 
\bit

\item
Take the set $\hat{S}=S\cap V(\hat{T})$.
Check that $\hat{S}$ is an LD-set of $\hat{T}$ and $|\hat{S}|\le |S|-1$.
Hence, $\gl(\hat{T}) \le \gl(T) -1$.

\item
Let $\hat{S}$ be an $\il$-set of $\hat{T}$.
Take the set $S=\hat{S}\cup \{z\}$
Check that  $S$ is an ILD-set of $T$ and $|S|=|\hat{S}|+1$.
Hence, $\il(T) \le \il(\hat{T})+2$.

\eit


\noindent
{\bf Subcase (5.2):}  
$x\not\in S$.
For every integer  $j \in  [h]$, let $\hat{T}_j$ denote the tree obtained from $T-\{x, y,z\}$ that contains vertex $w_j$.
Take the   forest $\ds \hat{T} = T-\{x, y,z\} = \bigcup_{j=1}^{h} \hat{T}_j$
(see Figure \ref{fig7} (5.2)).

\bit

\item
Take the set $\hat{S}=S\cap V(\hat{T})$.
Check that $\hat{S}$ is an LD-set of $\hat{T}$ and $|\hat{S}|\le |S|-1$.
Hence, $\gl(\hat{T}) \le \gl(T) -1$.

\item
Let $\hat{S}$ be an $\il$-set of $\hat{T}$.
Take the set $S=\hat{S}\cup \{x,z\}$
Check that  $S$ is an ILD-set of $T$ and $|S|=|\hat{S}|+2$.
Hence, $\il(T) \le \il(\hat{T})+2$.

\eit
Thus, in all cases:
$\il(T) \le \il(\hat{T}) +2 \le 2\gl(\hat{T})-2+2= 2\gl(\hat{T})  \le 2(\gl(T)-1) =2\gl(T)-2$.
\end{proof}
\end{theorem}

\begin{figure}[hbt]
  \centering
        \includegraphics[width=.85\textwidth]{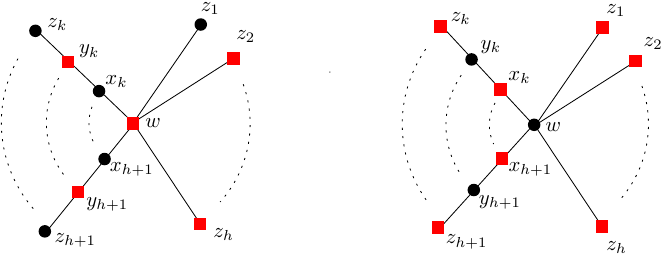}
  \caption{Tree of order $n=3k-2h+1$ with $k$ leaves, with  $2 \le h \le k$. 
  Left: the set of squared red  vertices  is a $\gl$-set.
  Right: the set of squared  red vertices is an $\il$-set.}
  \label{fig5}
\end{figure}

There exist trees for the whole range of values for $\gl$ and $\il$ of the previous theorem.

\begin{theorem}
Let $r,s$ a pair of integers such that $2 \le r \le s \le 2r-2$.
Then, there is a tree $T$ such that $\gl(T)=r$ and $\il(T)=s$.
\label{nildtrees2}
\begin{proof}
Let $h,k$ a pair of integers such that $2 \le h \le k$.
Take the tree $T$ of order $n=3k-2h+1$ with $k$ leaves, displayed in Figure \ref{fig5}.
Notice that the set $S_1=\{w\} \cup \{z_2, \ldots, z_h\} \cup \{y_{h+1}, \ldots, z_k\}$ is a $\gl$-set of $T$.
Observe also that  the set $S_2=\{z_1, \ldots, z_h\} \cup \{x_{h+1}, \ldots, x_k\} \cup \{z_{h+1}, \ldots, z_k\}$ is a 
$\il$-set of $T$.

Then, $\gl(T)=|S_1|=k$ and $\il(T)=|S_2|=2k-h$.
Take $h=2r-s$ and $k=r$.
Finally, check that $2 \le h \le k$ if and only if $2 \le r \le s \le 2r-2$.
\end{proof}
\end{theorem}

If a tree $T$ verifies some conditions, then we can obtain a better upper bound for $\il(T)$ than the one in Theorem~\ref{nildtrees1}.
\begin{prop}
\label{nildtrees4}
Let  $T$ be a tree of order $n \ge6$.
Then, $\il(T)\le 2\gl(T)-3$, whenever some of the following conditions holds.

\begin{enumerate}
\item 
$T$ contains a  strong support vertex $w$ such that $N(w)\cap {\cal L}(T)\ge3$.

\item
$T$ is a twin-free, i.e., if it has no strong support vertices. 

\end{enumerate}
\end{prop}

\begin{proof}
We proceed by induction on $n$.
According to Proposition \ref{small1}, for $n=6$,  $\il(T)\le \gl(T) \le  2\gl(T)-3$.
Let  $T$ be a tree of order $n \ge7$.
We distinguish cases.

\begin{enumerate}
\item 
Let $v\in N(w)\cap {\cal L}(T)$.
Take the subtree $\hat{T}$ of $T$ such that $V(\hat{T})=V(T)-v$.
Check that $\gl(\hat{T})=\gl(T)-1$ and $\il(\hat{T})=\il(T)-1$.

Thus,
$\il(T) = \il(\hat{T}) +1 \le 2\gl(\hat{T})-2+1= 2\gl(\hat{T})-1  = 2(\gl(T)-1)-1 =2\gl(T)-3$.

\item
We proceed by induction on $n$.
According to Proposition \ref{small1}, for $n=6$ this statement is true.
Let  $T$ be a tree of order $n \ge7$.
Proceeding as in the proof of Theorem \ref{nildtrees1}, we conclude that there is a subgraph $\hat{T}$ of $T$ satisfying the following two conditions:

\ben

\item[\bf (i)]
$\gl(\hat{T}) \le \gl(T) -1$

\item[\bf (ii)]
$\il(T) \le \il(\hat{T}) +2$

\een

\end{enumerate}

Thus,
$\il(T) = \il(\hat{T}) +2 \le 2\gl(\hat{T})-3+2= 2\gl(\hat{T})-1  = 2(\gl(T)-1)-1 =2\gl(T)-3$.
\end{proof}

\begin{figure}[hbt]
  \centering
        \includegraphics[width=.95\textwidth]{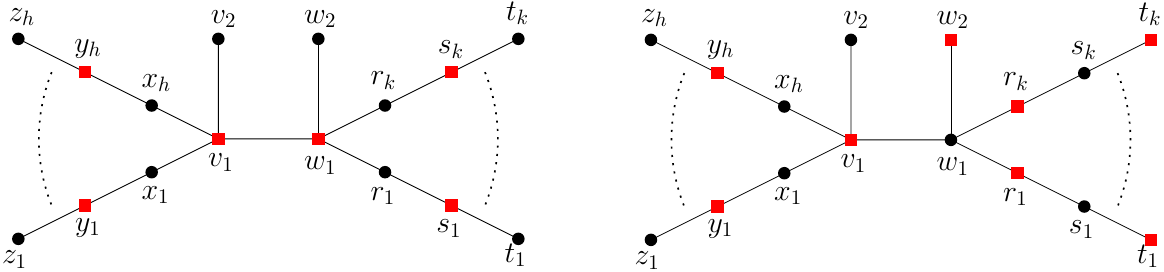}
  \caption{Tree of order $n=3h+3k+4$ with $h+k+2$ leaves, with $1\le h$ and $2 \le k$. 
  Left: the set of squared  vertices  is a $\gl$-set.
  Right: the set of squared  red vertices is an $\il$-set.}
  \label{fig12}
\end{figure}

Again, there are examples for this type of trees ranging from the spectrum of Proposition~\ref{nildtrees4}.

\begin{prop}
Let $r,s$ a pair of integers such that $2 \le r \le s \le 2r-3$.
Then, there is a twin-free tree $T$ such that $\gl(T)=r$ and $\il(T)=s$ (see Figure~\ref{fig12}).
\end{prop}
\label{nildtrees3}


\section{Unicyclic graphs}\label{unic}

Adding an extra edge to a tree, one obtains a unicyclic graph, but that edge changes everything. 
For example, unlike trees, not every unicyclic graph is an ILD graph. 
Figure \ref{fig6} shows the  smallest non-ILD unicyclic  graphs of girths 3 and 4, respectively.

\begin{figure}[hbt]
  \centering
        \includegraphics[width=0.85\textwidth]{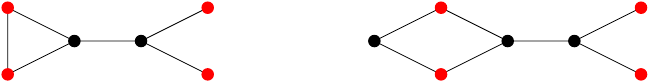}
  \caption{Unicyclic graphs of girth at most 4, containing no ILD-set.}
  \label{fig6}
\end{figure}

As an immediate consequence of the fact that in every twin-free bipartite graph (see Proposition \ref{bipild}), both stable sets are ILD-sets, the following result holds.

\begin{cor}
Let $v$ be a vertex of  a  twin-free bipartite graph $G$.
Then, there is an ILD-set of $G$ containing (resp., not containing) $v$.
\label{tft}
\end{cor}

And this result leads us to the following existence theorem for unicyclic ILD-graphs.

\begin{theorem}
\label{th:unicyclic}
Every twin-free unicyclic graph is an ILD graph.
\begin{proof}
According to Theorem \ref{ildg5} and Proposition \ref{bipild}, if $g\ge4$, then this result holds.
Let $G$ be a unicyclic graph of girth $g=3$.
Let  $C_3$ be its cycle in such a way that $V(C_3)=\{u_1,u_2,u_3\}$.
For every index $i\in[3]$, let $\deg(u_i)=d_i$.
Notice that, for every index $i\in[3]$, either $d_i=2$, or  $d_i=3$, $d_i\ge 4$.
If $d_i=3$, let $w_i$ denote the neighbor of $u_i$ not in $C_3$.

For every  $i\in[3]$, , whenever $d_i\ge 4$, let $T_i$ be the branching tree of $u_i$, i.e., the maximum subtree in $G$ containing $u_i$  such that  $V(T_i)\cap V(C_3)=\{u_i\}$.
For every index $i\in[3]$, whenever $d_i=3$, let $\hat{T}_i$ be the maximum subtree in $G$ containing $w_i$ such that   $V(\hat{T}_i)\cap V(C_3)=\emptyset$.

As an immediate consequence of Corollary \ref{tft}, we know that there is an ILD-set of  $T_i$ (resp., $\hat{T}_i$)    containing (resp., not containing) $u_i$ (resp., $w_i$).

We distinguish nine cases, assuming w.l.o.g. that $d_1 \le d_2 \le d_3$ and $d_3\ge3$.

\begin{figure}[hbt]
  \centering
        \includegraphics[width=0.8\textwidth]{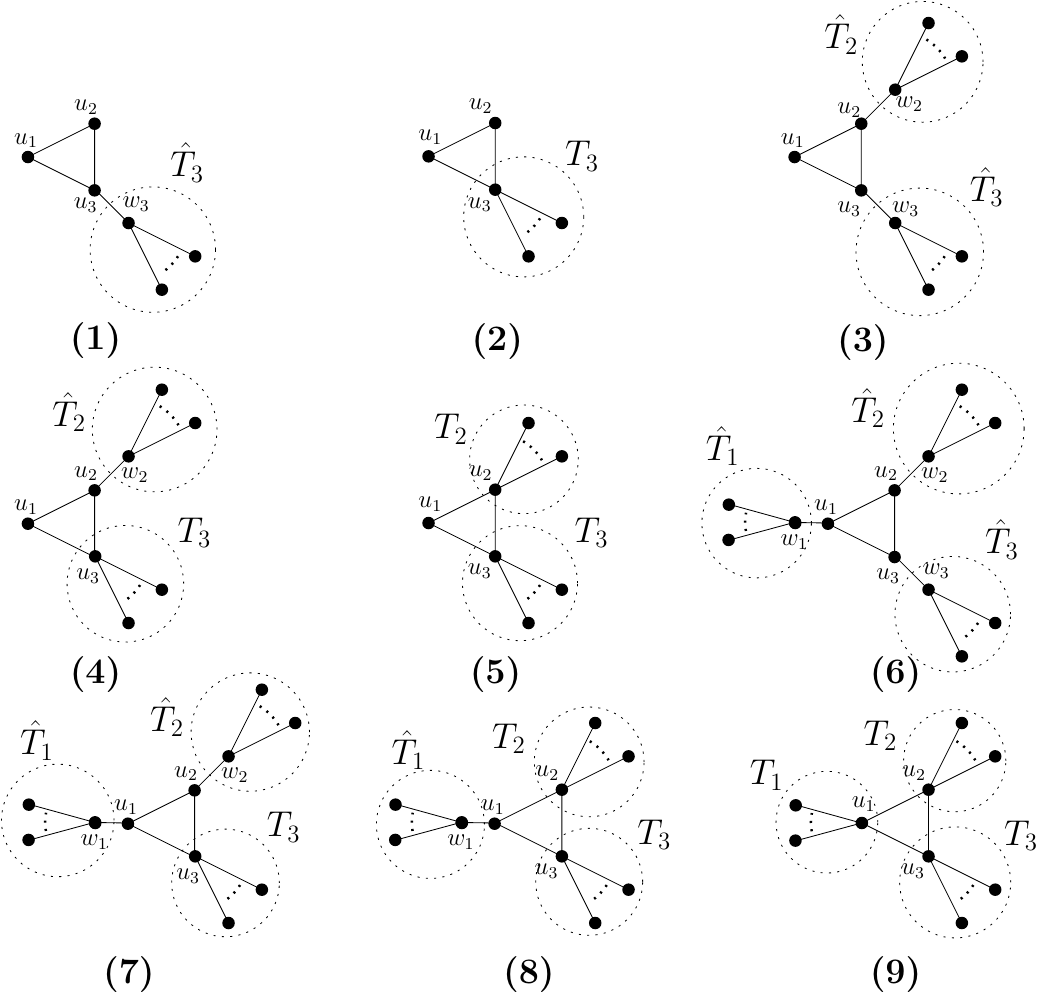}
  \caption{Nine twin-free unicyclic graphs of girth 3.}
  \label{fig14}
\end{figure}


\noindent
{\bf Case (1):} 
$d_1 = d_2=2$;  $d_3=3$. See Figure \ref{fig14} {\bf (1)}.
\newline
Take an ILD-set $\hat{S}_3$ of $\hat{T}_3$ containing vertex $w_3$.
Then, the set $S=\hat{S}_3\cup \{u_1\}$ is an ILD-set of $T$.


\noindent
{\bf Case (2):} 
$d_1 = d_2=2$;  $4 \le d_3$. See Figure \ref{fig14} {\bf (2)}.
\newline
Take an ILD-set $S_3$ of $T_3$ not containing vertex $u_3$.
Then, the set $S=S_3\cup \{u_1\}$ is an ILD-set of $T$.


\noindent
{\bf Case (3):} 
$d_1 = 2$;  $d_2=d_3=3$. See Figure \ref{fig14} {\bf (3)}.
\newline
Take an ILD-set $\hat{S}_2$ of $\hat{T}_2$ containing vertex $w_2$
and an ILD-set $\hat{S}_3$ of $\hat{T}_3$ containing vertex $w_3$.
Then, the set $S=\hat{S}_2\cup \hat{S}_3\cup \{u_1\}$ is an ILD-set of $T$.


\noindent
{\bf Case (4):} 
$d_1 =2$;  $d_2=3$;  $4 \le d_3$. See Figure \ref{fig14} {\bf (4)}.
\newline
Take an ILD-set $\hat{S}_2$ of $\hat{T}_2$ containing vertex $w_2$
and an ILD-set $S_3$ of $T_3$ containing vertex $u_3$.
Then, the set $S=\hat{S}_2\cup S_3$ is an ILD-set of $T$.


\noindent
{\bf Case (5):} 
$d_1 = 2$;  $4\le d_2\le d_3$. See Figure \ref{fig14} {\bf (5)}.
\newline
Take an ILD-set $S_2$ of $T_2$ not containing vertex $u_2$
andan ILD-set $S_3$ of $T_3$ not containing vertex $u_3$.
Then, the set $S=S_2\cup S_3\cup\{u_1\}$ is an ILD-set of $T$.


\noindent
{\bf Case (6):} 
$d_1 = d_2=d_3=3$. See Figure \ref{fig14} {\bf (6)}.
\newline
Take an ILD-set $\hat{S}_1$ of $\hat{T}_1$ not containing vertex $w_1$,
an ILD-set $\hat{S}_2$ of $\hat{T}_2$ containing vertex $w_2$
and an ILD-set $\hat{S}_3$ of $\hat{T}_3$ containing vertex $w_3$.
Then, the set $S=\hat{S}_1\cup\hat{S}_2\cup \hat{S}_3\cup \{u_1\}$ is an ILD-set of $T$.


\noindent
{\bf Case (7):} 
$d_1 = d_2=3$;  $4 \le d_3$. See Figure \ref{fig14} {\bf (7)}.
\newline
Take an ILD-set $\hat{S}_1$ of $\hat{T}_1$ not containing vertex $w_1$,
an ILD-set $\hat{S}_2$ of $\hat{T}_2$ containing vertex $w_2$
and an ILD-set $S_3$ of $T_3$ not containing vertex $u_3$.
Then, the set $S=\hat{S}_1\cup\hat{S}_2\cup S_3\cup \{u_1\}$ is an ILD-set of $T$.


\noindent
{\bf Case (8):} 
$d_1 = 3$;  $4\le d_2\le d_3$. See Figure \ref{fig14} {\bf (8)}.
\newline
Take an ILD-set $\hat{S}_1$ of $\hat{T}_1$ not containing vertex $w_1$.
an ILD-set $S_2$ of $T_2$ not containing vertex $u_2$
and an ILD-set $S_3$ of $T_3$ not containing vertex $u_3$.
Then, the set $S=\hat{S}_1\cup S_2\cup S_3\cup \{u_1\}$ is an ILD-set of $T$.


\noindent
{\bf Case (9):} 
$4\le d_1 \le d_2\le d_3$. See Figure \ref{fig14} {\bf (9)}.
\newline
Take an ILD-set $S_1$ of $T_1$ containing vertex $u_1$,
an ILD-set $S_2$ of $T_2$ not containing vertex $u_2$
and an ILD-set $S_3$ of $T_3$ not containing vertex $u_3$.
Then, the set $S=S_1\cup S_2\cup S_3$ is an ILD-set of $T$.
\end{proof}
\end{theorem}

\begin{itemize}

\item Non-ILD  unicyclic graphs of $g=3$ and order at most $n=10$.

\item Non-ILD  unicyclic graphs of $g=4$ and order at most $n=10$.

\item $\il(G) \le 2\gl(G) - 2$, for every twin-free unicyclic graph.

\item $\iota(G) \le \gl(G)$, for every twin-free unicyclic graph.

\end{itemize}

\section{Algorithms}\label{algo}

All the previous work allows us to get algorithms for finding an ILD-set for trees and unicyclic graphs, a linear one for the former and a quadratic time for the latter. 
The question of finding the minimum ILD-sets is still open.

We begin with the algorithm for trees.

\begin{algorithm}
\caption{Finding an ILD in a tree}
\label{algo1}
\begin{algorithmic}[1]
\Require A tree $T$.
\Ensure An ILD-set $S$.
\State Set $S:=\{\}$;
\While{there exists a subset of twins $D$ hanging from the vertex $u$}
\State $S:=S\cup D$;
\State Delete $D\cup \{u\}$ from $T$;
\EndWhile
\State Separate $T$ into its connected components, hence $T=\cup_{i=1}^k T_i$;
\State Find a 2-coloring of each $T_i$ and let $C_i$ be its smaller color class; 
\State $S:=S\cup (\cup_{i=1}^k C_i)$;
\end{algorithmic}
\end{algorithm}

\begin{theorem}
Algorithm~\ref{algo1} computes an ILD-set of a tree $T$ in linear time.
\end{theorem}
\begin{proof}
Clearly, if $T$ is a tree without twins, an ILD-set can be one of the color classes in a 2-coloring of the tree. 
It then  remains only to add, for each set of twins and its support vertex, the twins themselves to the ILD-set.

Regarding the algorithm's complexity, note that removing all twins from $T$ takes linear time. 
The subsequent steps, decomposing the tree into connected components and finding a 2-coloring for each component, also run in linear time.
\end{proof}

Now, we turn our attention to unicyclic graphs.

\begin{algorithm}[htbp]
\caption{Finding an ILD in a unicyclic graph without twins}
\label{algo2}
\begin{algorithmic}[1]
\Require A unicyclic graph $G$ without twins whose cycle is $C=\{u_1,\ldots ,u_n\}$ and $T_i$ is the maximal subtree in $G$ containing $u_i$.
\Ensure An ILD-set $S$ of $G$.
\State $S:=V(G)$; $\il=+\infty$;
\For {all each subtree $T_i$}
\State Find a 2-coloring of $T_i$ with color classes $C_0^i$ and $C_1^i$ where $u_i\in C_0^i$;
\EndFor
\For {each maximal independent set $I$ in $C$}
\State $\il^{\textrm{new}}:=\sum_{u_i\in I}|C_0^i|$;
\If{$\il^{\textrm{new}}<\il$}
\State $\il^{\textrm{new}}:=\il$;
\State $S:=\cup_{u_i\in I} C_0^i$;
\EndIf
\EndFor
\end{algorithmic}
\end{algorithm}

The following result checks the validity of the Algorithm~\ref{algo2}.

\begin{theorem}
Algorithm~\ref{algo2} computes an ILD-set of a unicyclic graph $G$ without twins and the whole process takes time $O(n^2)$ in the worst case.
\end{theorem}

\begin{proof}
The algorithm's correctness is established as follows. By Theorem~\ref{th:unicyclic}, an ILD-set for a unicyclic graph $G$ without twins consists of an ILD-set for each subtree $T_i$ and an ILD-set for the cycle $C$. For each subtree $T_i$, an ILD-set is given by one color class of its 2-coloring, since these subtrees contain no twins. For the cycle $C$, any maximal independent set is also an ILD-set. The algorithm exhaustively checks all valid combinations of these subsets.

The algorithm contains two main loops. The first iterates over the graph's subtrees, which are linear in number. The second iterates over the distinct maximal independent sets of $C$. Note that if $C$ is even, there are only two such sets; if $C$ is odd, there are $n$ maximal independent sets, where $n$ is the length of the cycle.

Finally, finding a 2-coloring for a tree requires linear time, as every tree is bipartite.
\end{proof}

\section{Conclusions, further work and acknowledgement}\label{cfw}

For general graphs, it was proved that there exist non-ILD graphs for any order $n\geq 9$. 
Focusing on trees, an improved upper bound for $\il$ both for the general case as well as for a special class of trees are given. 
In both situations, there are examples for the whole range of values of $\il$. 
The existence of ILD-sets in bipartite and unicyclic graphs has also been  established, however the problem remains open for other graph families. Proposition~\ref{prop:bipartite} about bipartite graphs is crucial to obtain the algorithms for finding ILD-sets both in trees and in unicyclic graphs.

There are plenty of work remaining in this area of study. As a matter of example, similarly as in Theorem~\ref{iogl}, it would be very interesting to characterize those trees where the bound established in Theorem~\ref{iogltrees} is reached.

\begin{conj}
Characterize those trees $T$ where $\iota(T)= \gl(T)$.
\label{conj1}
\end{conj}

Moreover, no characterization or other results on the ILD-sets of two big graph families such as bipartite graphs of girth 4 with twins, and twin-free graphs of girth 4 have been obtained yet. 
Finally, one can combine the notion of independence with other graph locations. 
Thus, we think that it is of interest to study independent identifying dominating sets (recall that and identifying-dominating set $D$ verifies that for every pair $u,v\in V(G)$: $N[u]\cap D \neq N[v]\cap D$), or independent metric-locating domination sets where $D$ is metric-locating-dominating if it is dominating and for every pair $u,v\in V(G)$ there exists $x\in D$ such that $d(u,x)\neq d(v,x)$.

The authors would like to thank M. Chellali for pointing out the proof of Theorem~\ref{iogltrees}.

\end{document}